\documentclass[reqno,11pt]{amsart}

\newtheorem{thm}{Theorem}[section]
\numberwithin{equation}{section}

\theoremstyle{remark}
\newtheorem{rem}[thm]{Remark}

\def\bint{{\ifinner\rlap{\bf\kern.35em--}
\int\else\rlap{\bf\kern.45em--}\int\fi}\ignorespaces}

\def\la{\langle}
\def\ra{\rangle}

\usepackage{amssymb, amsmath, amsthm}
\usepackage{import}
\usepackage{xifthen}
\usepackage{pdfpages}
\usepackage{inputenc}
\usepackage{fontenc}
\usepackage[english]{babel}
\newcommand{\R}{\mathbb{R}}
\def\dliminf{\displaystyle\liminf}

\newcommand{\id}{\mathrm{Id}}
\newcommand{\eps}{\varepsilon}
\usepackage{nicefrac, xfrac}

\begin{document}

\title[Transport maps between $\nicefrac{1}{d}$-concave densities]{On optimal transport maps\\ between $\nicefrac{1}{d}$-concave densities} 
\author{Guillaume Carlier}
\thanks{CEREMADE, UMR CNRS 7534, Universit\'e Paris
Dauphine, PSL, Pl. de Lattre de Tassigny, 75775 Paris Cedex 16, France and INRIA-Paris,
\texttt{carlier@ceremade.dauphine.fr}}
\author{Alessio Figalli}
\thanks{ETH Z\"{u}rich, Department of Mathematics, R\"{a}mistrasse 101, 8092, Z\"{u}rich, Switzerland,
\texttt{alessio.figalli@math.ethz.ch}}
\author{Filippo Santambrogio}
\thanks{Universite Claude Bernard Lyon 1, ICJ UMR5208, CNRS, \'Ecole Centrale de Lyon, INSA Lyon, Universit\'e Jean Monnet,
69622 Villeurbanne cedex, France, \texttt{santambrogio@math.univ-lyon1.fr}}
\date{\today}

\maketitle

\begin{abstract}

In this paper, we extend the scope of Caffarelli's contraction theorem, which provides a measure of the Lipschitz constant for optimal transport maps between log-concave probability densities in $\R^d$. Our focus is on a broader category of densities, specifically those that are $\nicefrac{1}{d}$-concave and can be represented as $V^{-d}$, where $V$ is convex. By setting appropriate conditions, we derive linear or sublinear limitations for the optimal transport map. This leads us to a comprehensive Lipschitz estimate that aligns with the principles established in Caffarelli's theorem.

\end{abstract}



\section{Introduction}

Given two probability densities $f$ and $g$ on $\R^d$ with finite second moments, the quadratic optimal transport problem between $\mu(\mbox{d} x)=f(x) \mbox{d}x$ and $\nu(\mbox{d} y)=g(y) \mbox{d}y$ consists in finding the change of variables $T$ between $\mu$ and $\nu$ that minimizes the mean squared displacement:
\[\int_{\R^d} \vert x-T(x)\vert^2 f(x) \mbox{d}x \]
subject to the constraint $T_\# \mu=\nu$ requiring that $T$ transports $\mu$ to $\nu$,  i.e., 
$$
\text{$\int_{T^{-1}(A)}  f(x) \mbox{d}x=\int_A g(y) \mbox{d}y\quad$ for every Borel subset $A$ of $\R^d$.}
$$
A  seminal work of Brenier \cite{Brenier} states that there exists a unique (up to negligible sets for $\mu$) solution $T$ to this quadratic optimal transport problem and that $T$  is characterized by the fact that it has a convex potential i.e. it is of the form $T=\nabla \varphi$ with $\varphi$ convex. Without assuming that $\mu$ and $\nu$ have finite second  moments, McCann \cite{McCann} extended Brenier's result (with a different strategy based on cyclical monotonicity arguments rather than on optimal transport). The Brenier-McCann map $T=\nabla \varphi$ with $\varphi$ convex, seen as the monotone change of variables between the absolutely continuous measures $\mu$ and $\nu$ on $\R^d$, is by now considered as a fundamental object which  solves in particular in some suitable weak sense the Monge-Amp\`ere equation:
\begin{equation}
\label{eq:MA}
\det( D^2 \varphi ) \, g(\nabla \varphi)= f.
\end{equation}
 The regularity theory for the monotone transport and the corresponding Monge-Amp\`ere equation,  initiated by Caffarelli \cite{Caffarelli92}, has stimulated an intensive line of research in the last thirty years (see the survey \cite{DePFig} and the textbook \cite{Figallibook}). In the case we shall consider here of densities $f$ and $g$ which are (at least) $C^{0, \alpha}_{\mathrm{loc}}(\R^d)$ with $1/f,1/g \in L^{\infty}_{\mathrm{loc}}(\R^d)$, it follows from \cite{CorderoFigalli,FigalliJhaveri} (see also \cite{ADM}) that $\varphi \in C^{2, \alpha}_{\mathrm{loc}}(\R^d)$ and  solves the Monge-Amp\`ere equation in the classical sense. Regarding global estimates, a pathbreaking result of Caffarelli \cite{Caffarelli} states that if $f=e^{-V}$ and $g=e^{-W}$ for some $C^{1,1}_{\mathrm{loc}}(\R^d)$ functions $V$ and $W$ with $D^2 V \leq \Lambda \, \id$ and $D^2 W \geq \lambda\, \id$ ($\lambda>0$, $0\leq \Lambda <+\infty$) then $T$ is globally Lipschitz with the explicit bound $\Vert D T \Vert_{L^{\infty}(\R^d)} \leq \sqrt{\Lambda/\lambda}$. We refer the interested reader to \cite{CFJ}, \cite{ColomboFathi}, \cite{Kolesnikov}, \cite{FGP}, and the references therein for extensions, alternative proofs, and applications.
 
 Since Caffarelli's argument uses in a crucial way the concavity of the logarithm of the determinant, it is really tempting to try to exploit the stronger fact that the determinant to the power $\nicefrac{1}{d}$ is also a concave function on the space of symmetric positive definite matrices, and to see whether this can be useful to go beyond the case of log-concave densities. More precisely, we shall consider densities of the form
\[f=V^{-d}, \qquad g=W^{-d}\]
 with $W$ typically convex (in which case, we will say that $g$ is $\nicefrac{1}{d}$-concave) and $V$ semiconcave. Throughout the paper, we will always assume that $V$ and $W$ are of class $C^{1,1}_\mathrm{loc}$, bounded from below away from $0$, and of course we impose the compatibility condition
 \[\int_{\R^d} \frac{1}{V^d}=\int_{\R^d}  \frac{1}{W^d}.\]
  Note that this setting allows for heavy-tailed densities which do not necessarily have finite  second moments. Under suitable additional assumptions, we will indeed obtain a bound on the Lipschitz constant of the monotone transport between such measures, in the spirit of Caffarelli's result.

\smallskip

As a first step of independent interest, we establish some linear growth estimates on the monotone transport.

 \begin{thm}\label{boundT0}
 Assume that, for some $p>1$, $V^{1/p}$ is Lipschitz  and
 \begin{equation}\label{hyp0}
 \dliminf_{ \vert y \vert \to \infty} \frac{ W(y)}{\vert y \vert^p}>0, \qquad  \dliminf_{ \vert y \vert \to \infty} \frac{\nabla W(y) \cdot y}{W(y)}>1. 
 \end{equation}
 Then, the monotone transport $T$ from $V^{-d}$ to $W^{-d}$ satisfies \begin{equation}\label{Tlinear}
 \vert T(x)\vert \leq C(1+ \vert x \vert), \qquad \forall \,x\in \R^d,
 \end{equation}
for some explicit constant $C>0$.
 \end{thm}

 Using the previous result together with the concavity of $\det^{\nicefrac{1}{d}}$ will enable us to obtain global Lipschitz bounds \`a la Caffarelli. More precisely, let
 \[ \langle x \rangle:= \sqrt{1 + \vert x \vert^2}, \qquad \forall \,x\in \R^d.\]
When $V$, $W$, and their derivatives satisfy suitable power-like growth conditions,  we can bound the Lipschitz constant of the monotone transport.
 \begin{thm}\label{boundDT0}
 Assume that there exist exponents $p\geq q>1$ and constants $\lambda>0$ and $\Lambda \geq 0$ such that 
 \begin{equation}\label{semiconvexconcave}
  D^2 W \geq \lambda   \langle \; \cdot  \; \rangle^{p-2} \, \id  \quad\mbox{ and }\quad D^2 V \leq \Lambda     \langle \; \cdot \;  \rangle^{q-2} \,  \id, \qquad \mbox{  a.e. on $\R^d$},
 \end{equation}
and that 
\begin{equation}\label{powerlike}
 \langle \, \cdot \, \rangle^{1-p}  \vert \nabla W \vert, \;  \langle \; \cdot \;  \rangle^{1-q}  \vert \nabla V \vert, \;  \frac{\langle \, \cdot \, \rangle^{q} }{V}, \mbox{ and } \; \frac{  \langle \, \cdot \, \rangle^{q-1}}{ 1+ \vert \nabla V  \vert} 
\quad \mbox{ belong to $L^{\infty}(\R^d)$.}
 \end{equation}
Then, the monotone transport $T$ from $V^{-d}$ to $W^{-d}$  is $K$-Lipschitz for some explicit constant $K$.    
  \end{thm}

In section \ref{sec-point}, we prove Theorem \ref{boundT0} 
and explore several of its variants.  Note that (possibly nonlinear) growth estimates were obtained by Colombo and Fathi \cite{ColomboFathi}, using alternative methods based on concentration inequalities. Section \ref{sec-LipT} is devoted  
to establishing limits on the Lipschitz constant of the monotone transport map. This discussion includes a range of relevant assumptions that cover Theorem \ref{boundDT0} as a special case.
 
 \section{Pointwise bounds on $T$}\label{sec-point}

 \subsection{Proof of Theorem \ref{boundT0}}

For $R\geq 1$, let us replace $W$ by $W_R$:
\[W_R(y)=\begin{cases} C_R W(y) \mbox{ if $\vert y \vert \leq R$}\\ +\infty \mbox{ otherwise}\end{cases}\]
where $C_R:=(\int_{B_R} W^{-d})^{\nicefrac{1}{d}}$  is the normalizing constant making $W_R^{-d}$ a probability density, so that $C_R\to 1$ as $R\to \infty$. Then, the optimal transport map sending $f=V^{-d}$ to $W_R^{-d}$ is globally bounded.  Our aim is to prove a linear estimate on such a map, independent of $R$. 

To simplify the notation, we shall still denote by $T=\nabla \varphi$ the optimal transport map sending $f=V^{-d}$ to $W_R^{-d}$.
Also, we remove the subscript $R$ from $W_R$, that therefore will be denoted by $W$.

Normalize $\varphi$ so that $\min \varphi=1$. Since 
$\varphi(x) \to \infty$ as $\vert x \vert \to \infty$ and $\vert \nabla \varphi\vert$ is bounded by $R$, the positive function $\frac{\vert \nabla \varphi\vert^2}{ 2 \varphi}$ tends to $0$ as $\vert x\vert \to \infty$, hence achieves its maximum\footnote{This is the only reason in the proof to approximate $W$ with $W_R$, so that the density $W_R^{-d}$ is supported inside $B_R$.} at some $\bar x$. Setting
$$
u:=\frac{|\nabla \varphi|^2}{2},\qquad M=\max \frac{u}{\varphi}
$$
 then, at a maximum point  $\bar x$,  we have 
\begin{equation}\label{maxusurfi}
\nabla u=M \nabla \varphi \quad \mbox{ and } \quad D^2 u \leq M D^2 \varphi.
\end{equation}
Writing the Monge-Amp\`ere equation \eqref{eq:MA}  in log form, namely
$$
\log\det D^2\varphi=d\log W\circ \nabla \varphi - d\log V,
$$ 
we differentiate it in the direction $e_k$ (note that it follows from  Theorem 1.1 in \cite{CorderoFigalli} that $\varphi \in C^{3, \alpha}_{\mathrm{loc}}$). Then, with the convention of summation over repeated indices, we get
\begin{equation}\label{differentiateonce}
B^{ij}\varphi_{ijk}=d\frac{W_i}{W}\circ \nabla \varphi \varphi_{ik}-d\frac{V_k}{V},\qquad B:=(D^2\varphi)^{-1}.
\end{equation}
Now, we multiply \eqref{differentiateonce} by $\varphi_k$ and sum over $k=1,\ldots,d$.
Using
$$
\varphi_{ijk}\varphi_k=u_{ij}-\varphi_{ik}\varphi_{jk}, \qquad B^{ij}\varphi_{ik}\varphi_{jk}={\rm tr}D^2\varphi=\Delta\varphi,
$$
and noting that, at $\bar x$, thanks to  \eqref{maxusurfi}, we have
$$
\varphi_{ik} \varphi_k=M \varphi_i, \qquad B^{ij}u_{ij}\leq MB^{ij} v_{ij}=M{\rm tr} \id=Md,
$$
we deduce the following inequality
$$
Md-\Delta\varphi\geq d\biggl(M \frac{\nabla W}{W}\circ\nabla\varphi \cdot \nabla \varphi -\frac{\nabla V}{V}\cdot \nabla \varphi\biggr).
$$
We now remark that, by  the arithmetic-geometric mean inequality and using again the Monge-Amp\`ere equation \eqref{eq:MA}, we have
\[\Delta \varphi\geq d \big(\det D^2\varphi\big)^{\nicefrac{1}{d}}=d \frac{W\circ \nabla \varphi}{V}.\]
 Thus, we arrive at
$$
M-\frac{W\circ \nabla \varphi}{V} \geq M \frac{\nabla W}{W}\circ\nabla\varphi \cdot \nabla \varphi -\frac{\nabla V}{V}\cdot \nabla \varphi,
$$
i.e.,
\begin{equation}\label{bornesurM}
M \left[\frac{\nabla  W(\nabla \varphi)} {W(\nabla \varphi)} \cdot{\nabla \varphi}-1 \right] \leq \frac{\nabla V \cdot \nabla \varphi-W(\nabla \varphi)}{V} \leq  \frac{\vert \nabla V\vert }{V}\vert \nabla \varphi \vert.
\end{equation}
 Now our assumption \eqref{hyp0} implies that there exist positive constants $R_0$, $\delta_0>0$ and $C_0$ such that whenever $\vert y \vert \geq R_0$, one has  
 \begin{equation}\label{conseqhyp0}
\frac{\nabla  W(y) \cdot y} {W(\nabla y)} \geq 1+\delta_0, \qquad W(y) \geq C_0 \vert y \vert^p.
\end{equation}
Note that, if $\vert  \nabla \varphi (\bar x)  \vert\leq R_0$,  using $\varphi \geq 1$, we have $M \leq \frac{R_0^2}{2}$.

If, on the contrary,  $\vert  \nabla \varphi (\bar x)  \vert \geq R_0$, the first inequality in \eqref{conseqhyp0} implies that the left-hand side of \eqref{bornesurM} is positive.
In particular, the second term in \eqref{bornesurM} must be positive, namely $W(\nabla \varphi(\bar x))<\nabla V (\bar x)\cdot\nabla \varphi(\bar x)$. This, combined with the second inequality in \eqref{conseqhyp0}, yields
\[C_0  \vert \nabla \varphi(\bar x) \vert^p \leq  W(\nabla \varphi(\bar x)) \leq \vert \nabla V (\bar x) \vert \vert  \nabla \varphi(\bar x) \vert,\]
hence
\[ \vert \nabla \varphi(\bar x) \vert \leq \frac{ \vert \nabla V(\bar x) \vert^{\frac{1}{p-1}} }{C_0^{\frac{1}{p-1}}}.  \]
Combining all these bounds with \eqref{bornesurM}, we finally obtain
\[\delta_0 M \leq \frac{\vert \nabla V(\bar x) \vert^{\frac{p}{p-1}} }{ C_0^{\frac{1}{p-1}} V(\bar x)}.\] 
Finally, since $V^{1/p}$ is Lipschitz,
\[ \frac{\vert \nabla V(\bar x) \vert^{\frac{p}{p-1}} }{V(\bar x)}\leq \Big(p \; \mathrm{Lip} \; V^{1/p} \Big)^{\frac{p}{p-1}}.\]
In conclusion, we have found a universal bound on $M$:
\[M\leq M_0:=\max \bigg\{ \frac{R_0^2}{2}, \frac{\Big(p \; \mathrm{Lip} \; V^{1/p} \Big)^{\frac{p}{p-1}}}{\delta_0 C_0^{\frac{1}{p-1}} } \bigg\}.\]
Now, using
$$
|\nabla \sqrt{\varphi}|=\frac{|\nabla\varphi|}{2\sqrt{\varphi}}=\sqrt{\frac{u}{2\varphi} },
$$
we deduce that $\sqrt{\varphi}$ is $\sqrt{\frac{M_0}{2}}$ Lipschitz. Hence, denoting by $x^*$ the point where $\varphi$ is minimal  (i.e., with our convention, $\varphi(x^*)=1$), for all $x\in \R^d$ we have
\[\sqrt{\varphi(x)}  \leq 1 + \sqrt{\frac{M_0}{2}} \vert x-x^*\vert \leq 1 + \sqrt{\frac{M_0}{2}} (\vert x\vert+ \vert x^*\vert),\]
therefore, since $\frac{u}{\varphi}\leq M_0$,
\begin{equation}\label{boundwithx*}
 \vert \nabla \varphi(x)\vert \leq  \sqrt{2 M_0} + M_0 \vert x^*\vert +M_0 \vert x \vert.
 \end{equation}
The proof will therefore be complete once we find a universal bound on $x^*$, which can be done  as in \cite{Caffarelli,CFJ}  by a simple mass balance argument that we recall for the sake of completeness.

Given $ \alpha \in \mathbb S^{d-1}$, define the cone $K_\alpha:=\{y\in \R^d \; \; \alpha \cdot y \geq \frac{1}{2} \vert y\vert\}$. Since $T(x^*)=0$, the monotonicity of $T$ implies the inequality $T(x)\cdot (x-x^*)\geq 0$ for every $x$. In particular,  if $T(x) \in K_{x^* / \vert x^*\vert}$ we have
\[T(x)\cdot x \geq T(x) \cdot x^* \geq \frac{1}{2} \vert x^*\vert \vert T(x),\vert \]
which implies $|x|\geq \frac12 |x^*|$. Hence, we obtain $T^{-1}(K_{x^* / \vert x^*\vert}\setminus\{0\}) \subset \R^d \setminus B_{\frac{\vert x^*\vert}{2}}$. Thanks to $T_\# f=g$, this yields
\[a(g):=\inf_{\alpha\in  \mathbb S^{d-1}}  \int_{K_\alpha} g \leq \int_{\R^d \setminus B_{\frac{\vert x^*\vert}{2}}}f.\]
Since $\int_{\R^d\setminus B_r}f\to 0$ as $r\to +\infty$, we get a bound on $\vert x^*\vert$ that only depends on $f$ and $g$:
\[ \vert x^*\vert \leq  2  \sup\bigg\{r>0 \; :  \; \int_{\R^d\setminus B_r}f \geq  a(g) \bigg\}.\]
Using this bound\footnote{Note that, since $1/g \in L^{\infty}_{\mathrm{loc}}$, one can bound from below $a(g)$  regardless of the radius $R\geq 1$ we used in the beginning to truncate $g$.} in \eqref{boundwithx*} completes the proof of Theorem \ref{boundT0}. 
 
 \begin{rem}
 Another way to bound $M$ from above is to come back to \eqref{bornesurM} and bound its right-hand side in  terms of $W^*$, the Legendre transform of $W$:
\[ M \left[\frac{\nabla  W(\nabla \varphi)} {W(\nabla \varphi)} \cdot{\nabla \varphi}-1 \right] \leq \frac{\nabla V \cdot \nabla \varphi-W(\nabla \varphi)}{V}\leq \frac{ W^*(\nabla V)}{V}.\]
Under our assumption, it is possible to prove the existence of a constant $C>0$ such that $W^*(\nabla V ) \leq C V$. Then, combining this bound with the second condition in \eqref{hyp0}, we reach the same conclusion as in Theorem \ref{boundT0}. 

While in this case both arguments give the same result, the method presented in the proof of Theorem \ref{boundT0} seems more flexible than the one discussed here. In particular, we shall use a variation of it in the proof of Theorems~\ref{thm:pq} and~\ref{thm:pol to gauss}.
 \end{rem}

 \subsection{Variants} We want to stress the fact that a similar proof as that of Theorem \ref{boundT0} can lead to sublinear estimates on the optimal transport map. The idea is to consider a maximum point of $\frac{\vert \nabla \varphi\vert^2}{2 H(\varphi)}$ for some well chosen positive, increasing and smooth function $H$ such that $H(t)\to +\infty$ as $t\to +\infty$. Up to replacing $W$ by $W_R$ as we did above (and normalizing again $\varphi$ by fixing its minimum to $1$), one can assume the maximum is achieved at some point $\bar x$ and the goal is to estimate
 \[M:= \max \frac{\vert\nabla \varphi\vert^2}{2 H(\varphi)}=  \frac{\vert\nabla \varphi (\bar x) \vert^2}{2 H(\varphi(\bar x))}. \]
 At $\bar x$, one has
 \[D^2 \varphi \nabla \varphi =M H'(\varphi) \nabla \varphi, \quad \varphi_{ijk} \varphi_k + \varphi_{ik} \varphi_{jk} \leq M[H'(\varphi) \varphi_{ij}+ H''(\varphi) \varphi_i \varphi_j].\]
Hence, as in the proof of Theorem \ref{boundT0}, multiplying \eqref{differentiateonce} by $\varphi_k$ and summing over $k$ gives (recall the relation $B=D^2 \varphi^{-1}$) 
 \[\begin{split}
 d\bigg(M H'(\varphi) \frac{\nabla W(\nabla \varphi) \cdot \nabla \varphi }{W(\nabla \varphi) } &    -\frac{\nabla V}{V}\cdot \nabla \varphi\bigg)= B^{ij} \varphi_{ijk} \varphi_k\\
& \leq B^{ij} \Big(M[H'(\varphi) \varphi_{ij}+ H''(\varphi) \varphi_i \varphi_j] -\varphi_{ik} \varphi_{jk} \Big) \\
&= d M H'(\varphi) + MH''(\varphi) B \nabla \varphi \cdot \nabla \varphi-\Delta \varphi\\
& \leq d M H'(\varphi) + \frac{ H''(\varphi)}{H'(\varphi)} \vert \nabla \varphi\vert^2 -d \frac{W (\nabla \varphi)}{V},
 \end{split}
 \]
 where, in the last line, we have used the arithmetic-geometric mean inequality and that, at $\bar x$, one has $ MD^2 \varphi^{-1} \nabla \varphi=\frac{\nabla \varphi}{H'(\varphi)}$. Since $\vert\nabla \varphi \vert^2= 2 M H(\varphi)$ at the point $\bar x$, we arrive at the inequality
 \begin{equation}\label{boundMavecH}
 M \bigg[ \frac{\nabla W(\nabla \varphi) \cdot \nabla \varphi }{W(\nabla \varphi)} -1 -\frac{ 2 H''(\varphi) H(\varphi)}{d H'(\varphi)^2}\bigg] \leq  \frac{\nabla V \cdot \nabla  \varphi-W(\nabla \varphi)}{ H'(\varphi) \;  V}. 
 \end{equation}
 An easy consequence of \eqref{boundMavecH} is the following:
 \begin{thm}
 \label{thm:pq}
 Assume that for some $p,q>1$ such that
 \begin{equation}\label{strangecondition}
 d(q-1)(p-1)>q-p
 \end{equation}
 one has
 \begin{equation}\label{hyp0p}
 \dliminf_{ \vert y \vert \to \infty} \frac{ W(y)}{\vert y \vert^p}>0, \qquad    \dliminf_{ \vert y \vert \to \infty} \frac{\nabla W(y) \cdot y}{W(y)}\geq p, 
 \end{equation}
 and, for some $A>0$,
 \begin{equation}\label{hyp0q}
 \frac{\vert \nabla V(x)  \vert}{V(x)} \leq \frac{A}{1+ \vert x\vert}, \quad \vert \nabla V(x)  \vert \leq A (1+\vert x\vert)^{q-1}, \qquad \forall \,x\in \R^d.
 \end{equation}
   Then  the monotone transport $T$ from $V^{-d}$ to $W^{-d}$ satisfies
 \[\vert T(x) \vert \leq C(1+ \vert x \vert)^{\frac{q-1}{p-1}}\qquad \forall \,x\in \R^d,\]
for some explicit constant $C>0$.
 \end{thm}
 
 \begin{proof}
 Let
\[\alpha=\frac{q-1}{p-1}, \qquad \theta=\frac{2 \alpha}{1+\alpha}=\frac{2 q-2}{p+q-2}<2. \]
Our aim is to bound
\[M:= \max \frac{ \vert \nabla \varphi \vert^2}{2 \varphi^\theta}.\]
For this, we choose the function $H(\varphi)=\varphi^{\theta}$ in \eqref{boundMavecH} so that, at the maximum point $\bar x$, we get
 \begin{equation}\label{majoMtheta}
 M \left[\theta \Big(\frac{\nabla  W(\nabla \varphi)} {W(\nabla \varphi)} \cdot{\nabla \varphi}-1\Big)+\frac{2-2\theta}{d} \right] \leq \varphi^{1-\theta} \frac{\nabla V \cdot \nabla \varphi-W(\nabla \varphi)}{V}.
 \end{equation}
As in the proof of Theorem \ref{boundT0}, it suffices to consider the case when $\vert \nabla \varphi (\bar x) \vert$ is large. In particular, by \eqref{hyp0p} and \eqref{strangecondition}, the prefactor in front of $M$ in the left-hand side of \eqref{majoMtheta} is strictly positive. This implies the nonnegativity of the right-hand side, that combined with \eqref{hyp0p} and \eqref{hyp0q} gives
 \begin{equation}\label{boundnablafialpha}
 \vert \nabla \varphi( \bar x)\vert  \lesssim  \vert \nabla V(\bar x) \vert^{\frac{1}{p-1}} \lesssim (1+ \vert \bar x \vert)^{\alpha}
 \end{equation}
 (here, the symbol $\lesssim$ is just hiding a positive multiplicative constant). Then, using \eqref{majoMtheta}, \eqref{hyp0q}, and the definition of $M$, we have
 \[\begin{split}
 M &\lesssim \frac{1}{1+ \vert \bar x \vert} \frac{ \vert \nabla \varphi (\bar x) \vert }{\varphi^{\frac{\theta}{2}} (\bar x)} \varphi^{1-\frac{\theta}{2}} (\bar x) \lesssim  \frac{M^{\frac{1}{2}}}{1+ \vert \bar x \vert} \frac{ \vert \nabla \varphi(\bar x)\vert^{\frac{2-\theta}{\theta}}}{M^{\frac{1}{\theta}-\frac{1}{2}}} \\
 & \lesssim M^{1-\frac{1}{\theta}} \frac{ (1+ \vert \bar x \vert)^{\frac{\alpha(2-\theta)}{\theta}  }}{1+ \vert \bar x \vert}  \lesssim M^{1-\frac{1}{\theta}}  \end{split}  \]
 where, in the second line, we used \eqref{boundnablafialpha} and $\frac{2-\theta}{\theta}=\frac{1}{\alpha}$.
 This gives a universal bound on $M$, and therefore on the Lipschitz constant of $\varphi^{1-\frac{\theta}{2}}$. So, arguing as in the end of the proof of Theorem \ref{boundT0} (using a bound on the norm of the minimum point of $\varphi$), we obtain
 \[ \vert T\vert = \vert \nabla \varphi \vert \lesssim \varphi^{\frac{\theta}{2}} \lesssim (1 + \vert  \cdot  \vert)^{\frac{\theta}{2-\theta}}= (1 + \vert   \cdot  \vert)^{\alpha},\]
 as desired.
 \end{proof}

 \begin{rem}
 It is unclear to us whether the condition \eqref{strangecondition} (that is automatically satisfied when $p\geq q$ or when $d$ is large) is really necessary to obtain a superlinear bound on $\vert T \vert$ when $p<q$.
 
 \end{rem}
 
 Using the same method, a similar  sublinear estimate can be obtained when the target is Gaussian:
 
 \begin{thm}
 \label{thm:pol to gauss}
 Assume that $g$ is a Gaussian and that there exists $A$ such that
  \begin{equation}\label{hyp0gq}
 \frac{\vert \nabla V(x)  \vert}{V(x)} \leq \frac{A}{1+ \vert x\vert}\qquad \forall \,x\in \R^d.
 \end{equation}
 Then the monotone transport $T$ from $V^{-d}$ to $g$ satisfies
  \[\vert T(x) \vert \leq C\sqrt{1+ \log(1+ \vert x \vert)}\qquad \forall \,x\in \R^d,\]
   for some explicit constant $C>0$.
 \end{thm}

 \begin{proof}
 Performing a translation and a suitable change of coordinates if necessary, me may assume that $W$ is of the form $W(y)=c_d e^{\frac{\vert y\vert^2}{2d}}$ (so that $g=W^{-d}$ is a standard Gaussian).
 This time, we take $H(\varphi):=\log(e+\varphi)$ and consider
 \[M:= \max \frac{ \vert \nabla \varphi \vert^2}{2 \log(e+\varphi)}.\]
 Note that the function $H$ is concave, so that \ref{boundMavecH} implis in this case
 \begin{equation}\label{boundMavecHconcave}
  M \bigg[ \frac{\nabla W(\nabla \varphi) \cdot \nabla \varphi }{W(\nabla \varphi)} -1\bigg] \leq  \frac{\nabla V \cdot \nabla  \varphi-W(\nabla \varphi)}{ H'(\varphi) \;  V}. 
  \end{equation}

 At a maximum point $\bar x$ (again such a point exists provided we replace $W$ by $W_R$ as in the proof of Theorem \ref{boundT0}), \eqref{boundMavecHconcave} becomes
 \[ M\left[\frac{2M}{d}\log(e+\varphi)-1\right]= M \left[\frac1d\vert \nabla  \varphi \vert^2 -1 \right] \leq (e+\varphi) \frac{\nabla V \cdot \nabla \varphi-W(\nabla \varphi)}{V}.\]
Again, we look for an explicit  bound on $M$. If $M \leq d$ there is nothing to prove. If $M>d$, the left-hand side is larger than $M\log(e+\varphi(\bar x)))$ (remember $\varphi\geq 1>0$) therefore, at the point $\bar x$, we have
  \begin{equation}\label{boundMgauss}
 M \leq \frac{ (e+\varphi)}{\log(e+ \varphi)} \frac{\nabla V \cdot \nabla \varphi-W(\nabla \varphi)}{V}.
  \end{equation}
As in the proof of Theorem \ref{boundT0}, the positivity of the right-hand side implies
 \[c_d e^{\frac{\vert \nabla \varphi(\bar x)\vert^2}{2d}} \leq \vert \nabla V(x) \vert  \vert \nabla \varphi(\bar x)\vert.\]
Thus, assuming without loss of generality that $\vert \nabla \varphi(\bar x)\vert$ is large (because if not, we automatically get a bound on $M$), the left-hand side is larger than $\vert \nabla \varphi(\bar x)\vert e^{\frac{\vert \nabla \varphi(\bar x)\vert^2}{4d}}$ and we get
 \[ \vert \nabla \varphi(\bar x)\vert  \leq 2d \sqrt{ \log ( \vert \nabla V(\bar x)\vert)} \lesssim  \sqrt{ 1+\log (1 + \vert \bar x\vert)}\]
 where, in the second inequality, we used the fact that \eqref{hyp0gq} implies a power growth condition on $\nabla V$.
 
 Denoting by $x^*$ the minimum point of $\varphi$ and arguing as in the last step of the proof of Theorem \ref{boundT0} in order to bound $|x^*|$, we get 
 \begin{equation}\label{logboundfi}
  \varphi(\bar x) \leq 1 + \vert \nabla \varphi(\bar x)\vert (\vert \bar x\vert + \vert x^*\vert) \leq C (1+\vert \bar x\vert) \sqrt{1+\log (1 + \vert \bar x\vert)}.
  \end{equation}
Also, by \eqref{boundMgauss}, the definition of $M$, and \eqref{hyp0gq}, we have
  \[M  \leq   \frac{  A(e+\varphi(\bar x))}{\log(e+ \varphi(\bar x)) \; (1+ \vert \bar x\vert) } \vert \nabla \varphi(\bar x) \vert  = \frac{A \sqrt{2 M}}{ (1+ \vert \bar x\vert)} \frac{ (e+\varphi(\bar x))}{\sqrt{\log(e+ \varphi(\bar x)) }}.\]
  Hence, combining this bound with \eqref{logboundfi} and the fact that $t\in [1,+\infty)\mapsto \frac{e+t}{\sqrt{\log(e+t)}}$ is increasing, we get
\begin{multline}
  \sqrt{M} \leq \frac{A \sqrt{2}}{ (1+ \vert \bar x\vert)} \frac{ (e+\varphi(\bar x))}{\sqrt{\log(e+ \varphi(\bar x)) }}\\
   \lesssim \frac{ \sqrt{1+\log (1 + \vert \bar x\vert)}}{\sqrt{ 1 + \log((1+\vert \bar x\vert) \sqrt{1+\log (1 + \vert \bar x\vert )})} } \lesssim 1,
  \end{multline}
  which gives the desired bound on $M$.
  
  Defining now 
  \[\Phi(t):=\int_1^t \frac{1} {\sqrt{\log(e+s)}} \mbox{d}s, \qquad t\geq 1,\]
  we see that the function $\Phi \circ \varphi$ is $\sqrt{2 M}$ Lipschitz. Hence
    \[\Phi(\varphi(x)) \leq  \sqrt{2M}(\vert x\vert + \vert x^*\vert) \leq C(1+ \vert x\vert), \qquad \forall \,x\in \R^d,\]
which implies
  \begin{equation}\label{boundgradinverse}
   \vert \nabla \varphi(x) \vert \leq  \sqrt{2M} \sqrt{\log\big(e+ \Phi^{-1}(C(1+ \vert x\vert))\big)}.
   \end{equation}
 Noticing that $\Phi^{-1}(0)=1$ and $(\Phi^{-1})'=\sqrt{\log(e+ \Phi^{-1})}  \leq \sqrt{\Phi^{-1}}$, one easily gets the subquadratic bound $\Phi^{-1}(t)\leq (1+\frac{1}{2}t)^2$, that together with \eqref{boundgradinverse} gives the desired estimate on $\vert T\vert$. 
   \end{proof}

 \section{Bounding the Lipschitz constant of $T$}\label{sec-LipT}
 
 \subsection{A formal argument}\label{par-formal} Assume that the optimal transport $T=\nabla \varphi$ from $V^{-d}$ to $W^{-d}$ has at most linear growth, i.e., it satisfies \eqref{Tlinear} for some constant $C$ (as is the case for instance under the assumptions of Theorem \ref{boundT0}). Also  assume that there are constants $A$, $B$, $\lambda>0$, and $\Lambda \ge 0$ such that
 \begin{equation}\label{hypAB}
 \frac{1+ \vert \nabla V(x)\vert^2}{V(x)^2} \leq \frac{A^2}{ \langle  x\rangle^2}, \quad \frac{W(x)^2} {1+ \vert \nabla W(x)\vert^2} \leq B^2 \langle  x\rangle^2, \qquad \forall\, x\in \R^d
 \end{equation}
 and, for a.e.  $x\in \R^d,$
 \begin{equation}\label{hyplambdaLambda}
 D^2 V(x) \leq  \frac{\Lambda(1+|\nabla V(x)|^2) \id }{V(x)},\qquad   D^2 W(x)  \geq \frac{ \lambda(1+|\nabla W(x)|^2) \id}{W(x)}.
 \end{equation}
 \begin{rem}\label{commentonfirstcondition}
 Since we have
\[ D^2\Big (-\frac{1}{V}\Big)=\frac{D^2 V }{V^2}-2 \frac{ \nabla V \otimes \nabla V}{V^3} , \]
we obtain
\[ \frac{D^2 V }{V^2}\geq D^2\Big (-\frac{1}{V}\Big) \geq \frac{D^2 V }{V^2}- 2 \frac{ \vert \nabla V \vert^2}{V^3} \id. \]
Thus, the first condition in  \eqref{hyplambdaLambda} amounts to a semiconcavity condition on $-\frac{1}{V}$ bounding from above its Hessian by a multiple of $ \frac{1+ \vert \nabla V \vert^2}{V^3} \id$. This explains the first condition that will appear in the sequel in \eqref{hyplambdaLambdaasymp}, which may look strange at first glance but is simply an incremental ratio version of this semiconcavity assumption. 
 \end{rem}
 
 We now show how,
 under assumptions \eqref{hypAB}-\eqref{hyplambdaLambda}, one can (at least formally) derive a bound on the Lipschitz constant of $T$, i.e., the largest eigenvalue of $D^2 \varphi$ by an argument \`a la Caffarelli. Then, in the next section we will show how to make the argument rigorous.
 
Let $\bar x, \bar e$ maximize the function $\R^d\times \mathbb S^{d-1}\ni(x,e)\mapsto \la D^2\varphi(x)e,e\ra$,
and assume, without loss of generality, $\bar e=e_1$. Then
 \begin{equation}\label{condmaxd2phi}
\nabla \varphi_{11}=0, \quad D^2 \varphi_{11} \leq 0, \quad \varphi_{i1}=0 \mbox{ for $i\neq 1$},
\end{equation}
where the last condition follows from the fact that $e_1$ has to be an eigenvector of $D^2 \varphi (\bar x)$. Writing the Monge-Amp\`ere equation in the form
\begin{equation}\label{mongeampere1:d}
F(D^2 \varphi)=\frac{W\circ \nabla \varphi}{V} \quad \mbox{ where $F=\det^{\nicefrac{1}{d}}$}
\end{equation}
and differentiating it first once and then twice  with respect to the $x_1$ variable, we get
\[F'(D^2 \varphi ) D^2 \varphi_1=\frac{W_i \varphi_{i1}}{V}-\frac{V_1W}{V^2},\]
and 
\begin{multline}
\label{eq:2nd der}
  \la F''(D^2 \varphi) D^2 \varphi_1, D^2 \varphi_1 \ra + F'(D^2 \varphi) D^2 \varphi_{11} =\frac{W_{ik}\varphi_{i1}\varphi_{k1}}{V} +\frac{W_i\varphi_{i11}}{V}\\
  -2\frac{W_i\varphi_{i1}V_1}{V^2}-\frac{V_{11}W}{V^2}+2\frac{V_1^2W}{V^3},
  \end{multline}
 where $W$  and its derivatives are evaluated at $\nabla \varphi$. Since $F$ is concave on the space of symmetric definite positive matrices, the first term in the left-hand side is nonpositive. Since,  at the point $\bar x$, $D^2 \varphi_{11} \le 0$ and $F$ is  nondecreasing  (in the sense of   matrices)   on the space of symmetric semidefinite positive matrices, the second term in the right-hand side of \eqref{eq:2nd der} is nonpositive at the point $\bar x$.  Thanks to \eqref{condmaxd2phi}, at the point $\bar x$ the right-hand side of  \eqref{eq:2nd der} simplifies to
\[\frac{W_{11}\varphi_{11}^2}{V} 
  -2\frac{W_1 V_1\varphi_{11}}{V^2}-\frac{V_{11}W}{V^2}+2\frac{V_1^2W}{V^3}.\]
Thus, setting $\bar y:=\nabla \varphi (\bar x)$, we deduce
\begin{multline}\label{crucialmaxd2fi}
W_{11} (\bar y) \varphi_{11}(\bar x)^2 \leq \frac{V_{11} (\bar x) W(\bar y)}{V(\bar x)} +\frac{2 V_1 (\bar x) W_1(\bar y) \varphi_{11}(\bar x)}{V(\bar x)}-\frac{ 2 V_1(\bar x) ^2 W(\bar y)}{V(\bar x)^2} \\
 \leq  \frac{V_{11} (\bar x) W(\bar y)}{V(\bar x)} + \frac{ W_{11} (\bar y) \varphi_{11}(\bar x)^2}{2}+\frac{2 V_1(\bar x) ^2 W_1(\bar y)^2}{V(\bar x)^2 W_{11}(\bar y)},
\end{multline}
where we have used Young's inequality in the last line and discarded the last (negative) term of the previous one. Therefore, we arrive at 
\[\frac{ W_{11} (\bar y) \varphi_{11}(\bar x)^2}{2} \leq \frac{V_{11} (\bar x) W(\bar y)}{V(\bar x)} + \frac{2 V_1(\bar x) ^2 W_1(\bar y)^2}{V(\bar x)^2 W_{11}(\bar y)}\]
which, thanks to \eqref{hyplambdaLambda}, yields
\begin{equation}\label{majorfi11square}
 \varphi_{11}(\bar x)^2 \leq \Big(\frac{2\Lambda}{\lambda} + \frac{4}{\lambda^2}\Big) \frac{1 + \vert \nabla V (\bar x) \vert^2}{V(\bar x)^2} \frac{W (\bar y)^2}{1 + \vert \nabla W (\bar y) \vert^2}.
 \end{equation}
Thus, \eqref{hypAB} and \eqref{Tlinear} imply that, denoting by $M=M( A, B, \lambda, \Lambda)$ the explicit constant 
\[M=\sqrt{\frac{2\Lambda}{\lambda} + \frac{4}{\lambda^2}} AB,\]
one has 
\[ \max_{x,e}\la D^2\varphi(x)e,e\ra=\varphi_{11}(\bar x) \leq M \frac{1 + \vert \bar  y\vert }{ 1+ \vert \bar x\vert}= M \frac{1 + \vert  \nabla  \varphi(\bar x) \vert }  { 1+ \vert \bar x\vert} \leq  M (1+C),\]
where $C$ is the constant in \eqref{Tlinear}.

This shows that $T=\nabla \varphi$ is $M(1+C)$ Lipschitz, as desired. Unfortunately, this argument is only formal not only because it assumes $\varphi$ to be four times differentiable (which anyhow could be assumed by approximation) but, more importantly, because maximizing sequences for the second derivatives of $\varphi$ need not be bounded. Hence, to give rigorous statements and proofs, we shall use incremental ratios (similarly to \cite{CFJ}) as explained in the next section.

 \subsection{Using incremental ratio} Given a function $h : \R^d \to \R$, a small constant $\eps>0$, and $(x,e)\in \R^d\times \mathbb S^{d-1}$, set
 \begin{equation}\label{defsecondfd}
 h^\eps(x,e):=h(x+\eps e)+h(x-\eps e)-2 h(x).
 \end{equation}
 For $R>0$ (chosen large in a way that will be specified later), we replace $W$ by $W_R$ as we did in the beginning of the proof of Theorem \ref{boundT0}. Then, by an abuse of notation, we denote the monotone transport map from $V$ to $W_R$ by $T=\nabla \varphi$. To obtain a Lipschitz bound on $T$, our goal is to find an upper bound that does not depend on $R$ on the quantity
 \[\limsup_{\eps\to 0^+} \; \eps^{-2}  \sup_{\R^d  \times   \mathbb S^{d-1}} \varphi^{\eps}.\]
 Indeed, this will enable us to conclude thanks to the stability of the monotone transport map as $R\to \infty$. Thanks to Lemma 3.1 in \cite{CFJ},  denoting $\hat x:=\frac{x}{\vert x\vert}$ for $x\in \R^d\setminus\{0\}$, one has
 \begin{equation}\label{asympcone}
    T(x)-R \hat{x} \to 0 \mbox{ as $\vert x \vert \to +\infty$}.
    \end{equation}
 This implies in particular that, for any fixed $\eps>0$, the positive function $\varphi^\eps(x,e)=\int_0^\eps (T(x+t\eps)-T(x-t\eps)) \cdot e \, dt$ tends to $0$ as $\vert x \vert \to \infty$, uniformly in $e\in  \mathbb S^{d-1}$. In particular, it achieves its maximum at some point $\bar x, \bar e$ (for the moment, we do not explicitly write the dependence of  $\bar x, \bar e$ in terms of $\eps$). Setting $\bar y:= \nabla \varphi (\bar x)$, the fact that $\bar x, \bar e$ maximizes $\varphi^\eps$ over $\R^d  \times   \mathbb S^{d-1}$ yields
 \begin{equation}\label{optiinx}
 \nabla \varphi(\bar x + \eps \bar e)+   \nabla \varphi(\bar x - \eps \bar e)=2 \bar y, \quad D^2 \varphi (\bar x + \eps \bar e)+   D^2 \varphi(\bar x - \eps \bar e)\leq 2 D^2 \varphi (\bar x) 
 \end{equation}
 and that, for some $\delta\in \R$, one has
 \begin{equation}\label{opitiine}
 \nabla \varphi(\bar x + \eps \bar e)-   \nabla \varphi(\bar x - \eps \bar e)=2 \delta \bar e.
 \end{equation}
Thus, using the notation $\partial_{\bar e} \varphi=\nabla \varphi \cdot \bar e$, we have
 \begin{equation}\label{deltaeps}
  \nabla \varphi(\bar x + \eps \bar e)=\bar y\pm \delta \bar e, \quad \delta=\frac{1}{2}(\partial_{\bar e} \varphi(\bar x + \eps \bar e)- \partial_{\bar e} \varphi(\bar x - \eps \bar e)).
  \end{equation}
  In particular, it follows from \eqref{deltaeps} that we have
\begin{multline}\label{fiepsdeltaeps} \sup_{\R^d  \times   \mathbb S^{d-1}} \varphi^{\eps}=
\varphi^\eps (\bar x ,\bar e)=\int_0^\eps \big( \partial_{\bar e} \varphi(\bar x+ t\bar e)-\partial_{\bar e} \varphi(\bar x -t\bar e)\big)\,dt \\
\leq \int_0^\eps \big( \partial_{\bar e} \varphi(\bar x+ \eps \bar e)-\partial_{\bar e} \varphi(\bar x -\eps\bar e)   \big)\,dt =2\eps \delta,
\end{multline}
where we have used the convexity of $\varphi$ to obtain
\begin{equation}\label{dpepsmon}
 \partial_{\bar e} \varphi(\bar x+ t\bar e) \leq  \partial_{\bar e} \varphi(\bar x+ \eps \bar e), \quad  \partial_{\bar e} \varphi(\bar x- t\bar e) \geq  \partial_{\bar e} \varphi(\bar x- \eps \bar e) ,\qquad \forall\,t \in (0,\eps).
\end{equation}
Hence, to prove $\varphi^\eps(\bar x, \bar e) \leq K\eps^2$ for some universal constant $K$, it will be sufficient to show $\delta \leq K\eps$ for some universal  constant $K$.

\smallskip

Denoting as before $F=\det^{\nicefrac{1}{d}}$, the second condition in \eqref{optiinx} yields
\[\begin{split}
F(D^2 \varphi(\bar x))  &\geq F\Big(\frac{1}{2}  D^2 \varphi (\bar x + \eps \bar e)+ \frac{1}{2}   D^2 \varphi(\bar x - \eps \bar e) \Big)\\
&\geq \frac{1}{2}  F\Big(  D^2 \varphi (\bar x + \eps \bar e)\Big)+ \frac{1}{2}  F\Big(  D^2 \varphi (\bar x - \eps \bar e)\Big)
\end{split}\]
where we used again the concavity of $F$. Using the Monge-Amp\`ere equation \eqref{mongeampere1:d} at the points $\bar x$, $\bar x \pm \eps \bar e$, and recalling \eqref{deltaeps}, we obtain
\[\begin{split}
0&\geq \frac{W(\bar y + \delta \bar e)}{V(\bar x+ \eps \bar e)}+ \frac{W(\bar y - \delta \bar e)}{V(\bar x- \eps \bar e)} -2 \frac{W(\bar y)} {V(\bar x)}\\
&=\frac{W(\bar y + \delta \bar e)+ W(\bar y - \delta \bar e)-2 W(\bar y)}{V(\bar x)} \\
&+ \Big(W(\bar y + \delta \bar e)- W(\bar y)\Big) \Big( \frac{1}{V(\bar x+ \eps \bar e)}-\frac{1}{V(\bar x)}\Big)   \\
&+\Big(W(\bar y - \delta \bar e)- W(\bar y)\Big) \Big( \frac{1}{V(\bar x- \eps \bar e)}-\frac{1}{V(\bar x)}\Big)\\
&+ W(\bar{y}) \Big( \frac{1}{V(\bar x+\eps \bar e)}+ \frac{1}{V(\bar x- \eps \bar e)}- \frac{2}{V(\bar x)} \Big).
\end{split}\]
Thus, multiplying by $V(\bar x)$ and using the notation in \eqref{defsecondfd}, we obtain a discrete analog of \eqref{crucialmaxd2fi}:
\begin{equation}\label{crucialdiscreteineq}
W^{\delta}(\bar y, \bar e) \leq \Gamma^+ + \Gamma^-+ V(\bar x)W(\bar y) \Big(-\frac{1}{V}\Big)^\eps( \bar x, \bar e),
\end{equation}
where 
\begin{equation}\label{deflambdapm}
 \Gamma^\pm:=(W(\bar y \pm \delta \bar e)- W(\bar y)) \frac{V(\bar x \pm \eps \bar e)-V(\bar x)}{V(\bar x \pm \eps \bar e)}.
\end{equation}
The goal now is to find suitable assumptions ensuring that \eqref{crucialdiscreteineq} implies  $\delta \leq K \eps$.
This is the purpose of the next section.

 \subsection{Assumptions ensuring that $T$ is globally Lipschitz}
We aim to find assumptions guaranteeing that \eqref{crucialdiscreteineq} implies a universal bound $\delta \leq K \eps$. Indeed, thanks to  \eqref{fiepsdeltaeps}, this will imply a uniform bound on the eigenvalues of $D^2 \varphi$.

We shall assume that $V$ and $W$ satisfy assumptions \eqref{hypAB} and \eqref{hyplambdaLambda}  that naturally appeared in the formal discussion in Section \ref{par-formal}. However, to address the case where the maximizer $\bar x= \bar x(\eps)$ escapes at $\infty$ as $\eps \to 0$, we shall need extra hypotheses that are to, some extent, asymptotic incremental ratio versions of  \eqref{hypAB} and \eqref{hyplambdaLambda} (recall Remark \ref{commentonfirstcondition}). Still, as we shall see, the constants from these extra assumptions do not appear in the final Lipschitz bound.
 
 These extra assumptions are as follows: There exist $A_0,B_0,\lambda_0,\Lambda_0,R_0,\alpha_0>0$ such that, for every $z\in \R^d\setminus B_{R_0}$, $e\in  \mathbb S^{d-1}$,  and $\alpha \in (0, \alpha_0)$, the following holds:
 \begin{equation}\label{hypABasymp}
\frac{ \vert V(z+\alpha e)-V(z)\vert }{ \alpha V(z+\alpha e)}  \leq \frac{ A_0  }{1+ \vert z \vert}, \quad  \frac{ \vert W(z+\alpha e)-W(z)\vert  W(z)}{ \alpha(1+ \vert \nabla W (z)\vert^2)}  \leq B_0  (1+ \vert z \vert),
 \end{equation}
 \begin{equation}\label{hyplambdaLambdaasymp}
  \alpha^{-2}\Big(-\frac{1}{V}\Big)^{\alpha}(z,e) \leq \frac{\Lambda_0(1+|\nabla V(z)|^2)  }{V(z)^3} ,\quad  \frac {W^{\alpha}(z, e)}{\alpha^2}  \geq \frac{ \lambda_0(1+|\nabla W(z)|^2)}{W(z)}.
  \end{equation}
We are now in position to prove: 
 
 \begin{thm}\label{boundlipgen}
In addition to the assumptions of Theorem \ref{boundT0}, let $V$ and $W$  satisfy \eqref{hypAB}-\eqref{hyplambdaLambda}-\eqref{hyplambdaLambdaasymp}-\eqref{hypABasymp} for some constants  $A,B,\lambda,\Lambda> 0$, and $A_0,B_0,\lambda_0,\Lambda_0,R_0,\alpha_0$. Then, the monotone transport map $T=\nabla \varphi$ from $V^{-d}$ to $W^{-d}$ satisfies
 \begin{equation}\label{boundDT}
 DT(x)=D^2 \varphi(x) \leq K \, \id,   \qquad \forall \,x\in \R^d, 
 \end{equation}
with $K:=\sqrt{\frac{2\Lambda}{\lambda} + \frac{4}{\lambda^2}}AB (1+C)$, where  $C$ is the constant in \eqref{Tlinear} obtained in the proof of Theorem \ref{boundT0}.
 \end{thm}

 \begin{proof}
 Recall that we have truncated the target measure $g$ by replacing $W$ with $W_R$. This slightly affects the value of constants $\lambda$ and $B$ in the assumptions concerning $W$,
 but these values will converge to the ones of $W$ when in the end one let $R\to \infty$.
 Note also that, since $W^{-d}$ is supported inside $B_R$, the arguments of $W$ and its derivatives will also belong to $B_R$.
In addition, as in the proof of Theorem \ref{boundT0}, we have a linear bound on $T=T_R$, with a constant which converges to the value $C$ in \eqref{Tlinear} as $R\to \infty$. Hence, to prove \eqref{boundDT}, it is enough to show the same inequality for $T_R$ with these modified constants depending on $R$, and then to let $R\to +\infty$. 
For this reason and to simplify the notation, in what follows we omit the dependence on $R$. On the contrary, we shall from now on explicitly write the dependence of $\bar x$, $\bar e$ $\bar y=\nabla \varphi(\bar x)$, and $\delta$ with respect to $\eps,$ denoting them, respectively, $\bar x(\eps)$,  $\bar e (\eps)$, $\bar y(\eps)$, and $\delta(\eps)$. 
We will need to consider two cases.
  
  \smallskip
 
 \noindent  
{$\bullet$ \textbf{First case: $\bar x(\eps)$ has a limit point as $\eps \to 0^+$.}} 
In this case, up to regularizing $V$ and $W$ (which only slightly changes that value of the constants in the assumptions, and that anyhow will converge to the original values as the regularizing parameter goes to zero, we can assume that $V$ and $W$ are both of class $C^{2,\alpha}_\mathrm{loc}$ for some $\alpha>0$. This guarantees, in particular, that the function $\varphi$ is at least of class $C^4$. 

Now, up to a suitable extraction, we may assume that $\bar x(\eps)$ converges to some $\bar x$ and that $\bar e(\eps)$ converges to a unit vector, say $e_1$. Then, it is easy to check that
\begin{equation}\label{ptlimmaxd2fi}
\varphi_{11}(\bar x)=\max_{(x,e)\in \R^d \times \mathbb S^{d-1}}  \la D^2\varphi(x)e,e\ra.
\end{equation}
Thus, in this case, the formal argument described in Section~\ref{par-formal} can be applied verbatim
to deduce that
$$
\varphi_{11}(\bar x)\leq \sqrt{\frac{2\Lambda}{\lambda} + \frac{4}{\lambda^2}}AB (1+C),
$$
proving the desired estimate.

 \smallskip
 
  \noindent  
{$\bullet$ \textbf{Second case: $\vert \bar x(\eps)\vert\to \infty$ as $\eps \to 0^+$.}} Note first that, thanks to \eqref{asympcone} and \eqref{opitiine},
$$
\delta(\eps)=\frac{1}{2} \vert T (\bar x(\eps)+ \eps \bar e (\eps))- T (\bar x(\eps)- \eps \bar e (\eps))\vert\to 0
\qquad\text{as $\eps \to 0^+$.}
$$
Also, again by \eqref{asympcone},
$$
\vert \bar y(\eps)\vert=\vert T(\bar x(\eps))\vert \to R \qquad\text{as $\eps\to 0^+$.}
$$
Hence, up to choosing $R$ large enough, we may assume  that for $\eps$ small enough one can use the asymptotic assumptions \eqref{hypABasymp}-\eqref{hyplambdaLambdaasymp} concerning $W$  (respectively $V$) at  $(z, \alpha, e)=(\bar y(\eps), \delta(\eps), \pm \bar e (\eps))$ (resp. $(\bar x(\eps), \eps, \pm \bar e (\eps))$.
Thus, for $\eps>0$  small,  it follows from \eqref{crucialdiscreteineq}, \eqref{hypABasymp}, and \eqref{hyplambdaLambdaasymp},
that 
 \[\lambda_0 \delta(\eps)^2 \leq 2   A_0B_0 \frac{1+ \vert \bar y(\eps)\vert} {1+ \vert \bar x(\eps)\vert} \eps \delta(\eps)  + \Lambda_0 \frac{1 + \vert \nabla V (\bar x(\eps)) \vert^2}{V(\bar x(\eps))^2} \frac{W (\bar y(\eps))^2}{1 + \vert \nabla W (\bar y(\eps) \vert^2}  \eps^2. \]
Combining this bound with \eqref{hypAB},
we obtain
  \[\lambda_0 \delta(\eps)^2 \leq 2   A_0B_0 \frac{1+ \vert \bar y(\eps)\vert} {1+ \vert \bar x(\eps)\vert} \eps \delta(\eps)  + \Lambda_0 A^2B^2 \frac{1+ \vert \bar y(\eps)\vert^2} {1+ \vert \bar x(\eps)\vert^2} \eps^2. \]
Since $\vert \bar x(\eps)\vert\to \infty$ and $\vert \bar y(\eps)\vert\to R$ as $\eps\to 0^+$, this implies
\[\delta(\eps) =o(\eps).\]
Recalling \eqref{fiepsdeltaeps}, this implies  that $D^2 \varphi \equiv 0$, which would prove that the map $T$ is constant, and is a contradiction (since $g$ is not a Dirac mass). 
Hence, this second case cannot happen, which concludes the proof.
 \end{proof}
 
 \subsection{Proof of Theorem \ref{boundDT0}.} It suffices to check that the assumptions of Theorem \ref{boundDT0} imply those of Theorem \ref{boundlipgen}. 
 
 We start from some preliminary considerations. First, note that, by \eqref{semiconvexconcave}, $W$ is convex. Thus, combining \eqref{semiconvexconcave} and the asssumption $\vert \nabla W \vert \lesssim \langle \, \cdot \, \rangle^{p-1}$, it follows that $ W  \langle \, \cdot \, \rangle^{-p} $ is bounded both from above and away from $0$ (which we shall simply denote as $W \sim \langle \, \cdot  \,\rangle^{p}$). 
 
  Using  Taylor's formula with an integral remainder, \eqref{semiconvexconcave}  yields
 \begin{eqnarray}\label{TaylorW} \langle y \rangle \vert \nabla W(y) \vert \geq \nabla W(y)\cdot y &=&W(y)-W(0) + \int_0^1 s \langle D^2 W(sy) y,  y  \rangle \mbox{d} s\notag\\
 & \geq& W(y)-W(0)+\frac{\lambda \vert y \vert^p}{p} \end{eqnarray}
 from which we deduce $(1+\vert \nabla W \vert) \sim \langle \, \cdot \, \rangle^{p-1}$.

 We list here below the different assumptions that we need to prove, and explain how to obtain them.
 \begin{itemize}
  \item Lipschitz condition on $V^{1/p}$ (an assumption in Theorem \ref{boundT0}): the fact that $V^{\frac{1}{q}}$ (hence also $V^{\frac{1}{p}}$, since $p\geq q$ and $V$ is bounded away from $0$) is Lipschitz directly follows from the second and third conditions in \eqref{powerlike}. 

 \item First condition in \eqref{hyp0} (an assumption of Theorem \ref{boundT0}): we already pointed out that we have $W \sim \langle \, \cdot  \,\rangle^{p}$; in particular, $W$ satisfies the first condition in \eqref{hyp0}.

 \item Second condition in \eqref{hyp0} (an assumption in Theorem \ref{boundT0}): this is a consequence of \eqref{TaylorW}, dividing it by $W(y)$.
 
 \item Upper bound on $\frac{W(x)^2} {1+ \vert \nabla W(x)\vert^2} $ (from \eqref{hypAB}): we saw that from \eqref{TaylorW} we obtain $(1+\vert \nabla W \vert) \sim \langle \, \cdot \, \rangle^{p-1}$; this, together with $W \sim \langle \, \cdot  \,\rangle^{p}$, provides the desired bound.
 
 \item Upper bound on $  \frac{1+ \vert \nabla V(x)\vert^2}{V(x)^2}$  (from \eqref{hypAB}): this is a consequence of the second and third conditions in \eqref{powerlike}.

\item Upper bound on $D^2V$ (first condition in \eqref{hyplambdaLambda}): one deduces from \eqref{powerlike} that $V \sim \langle \, \cdot \, \rangle^{q}$; therefore, thanks to the last condition in \eqref{powerlike}, we have $\langle \, \cdot \, \rangle^{q-2} \lesssim \frac{1 + \vert \nabla V \vert^2}{V}$. The result then follows by assumption \eqref{semiconvexconcave}.


\item Lower bound on $D^2W$ (second condition in \eqref{hyplambdaLambda}): this is a consequence of the assumption \eqref{semiconvexconcave}, together with $(1+\vert \nabla W \vert) \sim \langle \, \cdot \, \rangle^{p-1}$ and $W \sim \langle \, \cdot  \,\rangle^{p}$.

\item Upper bound on the second-order incremental ratio of $-1/V$ (first condition in \eqref{hyplambdaLambdaasymp}):
Let $e\in  \mathbb S^{d-1}$, $\alpha \in [-1,1]$, and $z\in \R^d$ with $\vert z\vert$ large. Since $-D^2 \Big(\frac{1}{V}  \Big)\leq \frac{D^2 V}{V^2}$,  using \eqref{semiconvexconcave} and $V \sim \langle \, \cdot \, \rangle^{q}$, we get
 \[\Big(-\frac{1}{V}\Big)^{\alpha}(z,e) \leq \Lambda \int_0^\alpha \int_{-s}^s \frac{ \langle z+ \tau e \rangle^{q-2}} {V^2(z+\tau e)} \mbox{d} \tau \mbox{d} s \lesssim \langle  z \rangle^{-q-2}  \alpha^2 \]
 which implies  the first condition in 
  \eqref{hyplambdaLambdaasymp} since $\langle \, \cdot \, \rangle^{-q-2} \lesssim \frac{1 + \vert \nabla V \vert^2}{V^3}$. 
  
  \item Lower bound on the second-order incremental ratio of $W$ (second condition in \eqref{hyplambdaLambdaasymp}): this is essentially the same argument as the previous one (note, however, that the constants $\Lambda$ and $\lambda$ in  \eqref{hyplambdaLambdaasymp} may differ from those in \eqref{semiconvexconcave}.
  
  \item Upper bounds on the incremental ratios of $V$ and $W$ (the two conditions in \eqref{hypABasymp}): for large $|z|$, these two conditions can be deduced from upper bounds on $|\nabla V|/V$ and $|\nabla W|W/(1+|\nabla W|^2)$, respectively, and these two quantities are bounded thanks to the assumptions \eqref{powerlike} (in what concerns $V$) and to the fact that \eqref{TaylorW} provides $(1+\vert \nabla W \vert) \sim \langle \, \cdot \, \rangle^{p-1}$ and $W \sim \langle \, \cdot  \,\rangle^{p}$.
 \end{itemize}

\bigskip
  
{\bf Acknowledgments:}  The authors  acknowledge the support of the Lagrange Mathematics and Computing Research Center.

\bibliographystyle{plain}

\bibliography{bibli}

 \end{document}